\documentclass[12pt,a4paper]{amsart}
\usepackage[utf8]{inputenc}
\usepackage{amsbsy}
\usepackage{amscd} 
\usepackage{amsfonts}
\usepackage{amsgen} 
\usepackage{amsmath}
\usepackage{amsopn} 
\usepackage{amssymb}
\usepackage{amstext}
\usepackage{amsthm} 
\usepackage{amsxtra}
\usepackage{enumerate}
\usepackage{mathtools}
\usepackage{faktor}

\usepackage[colorlinks = true, linkcolor = blue, urlcolor  = blue, citecolor = red]{hyperref}

\theoremstyle{plain} 
\newtheorem{thm}{Theorem}[section]
\newtheorem{prop}[thm]{Proposition}
\newtheorem{lem}[thm]{Lemma}
\newtheorem{cor}[thm]{Corollary}

\theoremstyle{definition}
\newtheorem{defn}[thm]{Definition}
\newtheorem{rem}[thm]{Remark}
\newtheorem{ex}[thm]{Example}

\numberwithin{equation}{section}

\renewcommand{\theta}{\vartheta}
\renewcommand{\phi}{\varphi}
\renewcommand{\epsilon}{\varepsilon}
\renewcommand{\subset}{\subseteq}

\newcommand{\N}{\mathbb N}
\newcommand{\Z}{\mathbb Z}

\newcommand{\C}{\mathbb C}

\DeclareMathOperator{\id}{id}
\DeclareMathOperator{\Aut}{Aut}
\DeclareMathOperator{\QAut}{QAut}
\DeclareMathOperator{\GL}{GL}
\DeclareMathOperator{\PP}{\mathbb{P}}
\DeclareMathOperator{\E}{\mathbb{E}}
\DeclareMathOperator{\Var}{\mathbb{V}ar}

\begin{document}
\title[Almost all trees have quantum symmetry]{Almost all trees have quantum symmetry}

\author{Luca Junk}
\author{Simon Schmidt}
\author{Moritz Weber}
\address{Saarland University, Fachbereich Mathematik, Postfach 151150, 66041 Saar\-brü\-cken, Germany}
\email{junk@math.uni-sb.de}
\email{simon.schmidt@math.uni-sb.de}
\email{weber@math.uni-sb.de}
\thanks{The second and third authors were supported by the DFG project \emph{Quantenautomorphismen von Graphen}. The third author was supported by the SFB-TRR 195. This work was part of the first author's Bachelor's thesis written under the supervision of the other two authors.}

\date{\today}

\begin{abstract}
From the work of Erd\H{o}s and R\'{e}nyi from 1963 it is known that almost all graphs have no symmetry. In 2017, Lupini, Man\v{c}inska and Roberson proved a quantum counterpart: Almost all graphs have no quantum symmetry. Here, the notion of quantum symmetry is phrased in terms of Banica's definition of quantum automorphism groups of finite graphs from 2005, in the framework of Woronowicz's compact quantum groups. Now, Erd\H{o}s and R\'{e}nyi also proved a complementary result in 1963: Almost all trees do have symmetry. The crucial point is the almost sure existence of a cherry in a tree. But even more is true: We almost surely have two cherries in a tree - and we derive that almost all trees have quantum symmetry. We give an explicit proof of this quantum counterpart of Erd\H{o}s and R\'{e}nyi's result on trees.
\end{abstract}

\maketitle

\section{Introduction and main results}

Erd\H{o}s and R\'{e}nyi \cite{Erdos1963} proved that almost all graphs are asym\-met\-ric, in the following sense: When choosing a graph on $n$ vertices uniformly at random, the probability that its automorphism group is trivial, tends to $1$ as $n$ tends to infinity. In contrast to this, they also showed that almost all trees are symmetric, i.e.~the probability that the automorphism group of a random tree on $n$ vertices is trivial, tends to $0$ as $n$ tends to infinity.

In recent years, the notion of a quantum automorphism group of a graph was introduced by Banica in \cite{banica2005quantum}, modifying a preceding version by Bichon \cite{bichon2003quantum}. It is a compact matrix quantum group in the sense of Woronowicz (see \cite{woronowicz1987compact}) enlarging the usual automorphism group of a graph and providing a subtler graph invariant. It is an interesting question, which graphs have a quantum automorphism group that is strictly larger than their classical automorphism group. In that case we say that the graph has \textit{quantum symmetry}. In general, this question is very hard to answer. There are some sufficient criteria that one can check but usually, there is no 'easy' way to decide whether or not a graph has quantum symmetry. For results about specific graphs or specific classes of graphs see \cite{banica2007quantum}, \cite{schmidt2018petersen}, \cite{schmidt2018quantum}.

Recently, Lupini, Man\v{c}inska and Roberson proved a quantum version of the first Erd\H{o}s R\'{e}nyi Theorem mentioned in the beginning: Almost all graphs are quantum asymmetric, i.e.~the quantum automorphism group of almost all graphs is trivial (\cite{lupini2017nonlocal}). A crucial ingredient in their proof is the so-called \textit{coherent algebra} of a graph. They show that it provides a new sufficient criterion for a graph to be quantum asymmetric. We will give a few details of this in Section 4.

In this paper we give a proof of a quantum analogue of the second Erd\H{o}s R\'{e}nyi Theorem mentioned above, namely: Almost all trees have quantum symmetry. The crucial ingredient of Erd\H{o}s and R\'{e}nyi's proof is to show that almost all trees contain a so-called \textit{cherry} (two vertices of degree one 'dangling' like a twin cherry on a common third vertex) -- indeed, by flipping this twin cherry, we obtain a non-trivial automorphism. Interestingly, for proving the quantum analogue of this, it is enough to show that almost all trees contain at least \textit{two} cherries from which we may derive the existence of two non-trivial automorphisms with disjoint supports. By applying a criterion of one of the authors of this article \cite{schmidt2018foldedcube}, this yields the result.

We summarize:
\newtheorem*{MThm}{Theorem}
\begin{MThm}\label{main_thm} \ \newline \vspace*{-\baselineskip} \enlargethispage{\baselineskip}
	\begin{enumerate}[(i)]
		\item \cite{Erdos1963} Almost all graphs have no symmetry.
		\item \cite{lupini2017nonlocal} Almost all graphs have no quantum symmetry.
		\item \cite{Erdos1963} Almost all trees have symmetry.
		\item Almost all trees have quantum symmetry.
	\end{enumerate}
\end{MThm}

Let us note that the almost sure existence of two cherries in a tree is probably well-known to experts. Nevertheless, we give a direct proof of this fact in order to present a complete and concise proof of the quantum version of Erd\H{o}s and R\'{e}nyi's result on trees.

\section{Preliminaries}
In this section we give some mathematical background on compact (matrix) quantum groups, graphs and their (quantum) symmetries.

\subsection{Graphs and Trees}

A \textbf{graph} is a tuple $\Gamma = (V, E)$ where $V$ is a non-empty set of vertices and $E \subseteq V\times V$ is a set of edges (in particular, we don't allow a graph to have multiple edges between the same pair of vertices). It is called \textbf{finite} if $V$ is finite, and \textbf{undirected} if we have $(i, j) \in E$ whenever $(j, i) \in E$. An edge of the form $(i,i) \in E$ is called a \textbf{loop}. If $v \in V$ is a vertex, we define its \textbf{degree} to be the number of neighbours of $v$, i.e.~the number of vertices $u \in V$ such that $(v,u) \in E$.

We will only be concerned with finite undirected graphs without loops and without multiple edges. Furthermore, we will usually identify the vertex set $V$ with the set $\{ 1, \ldots, n\}$ where $n = |V|$.

A \textbf{walk} of length $k$ in a graph $\Gamma = (V, E)$ is a $k$-tuple $(e_1, \ldots, e_k)$ of edges $e_j = (u_j, v_j) \in E$ such that $v_j = u_{j+1}$ for all $j=1,\ldots k-1$. It is called a \textbf{cycle} if $u_1 = v_k$.\\
A graph is called \textbf{connected} if for every pair of vertices $i, j \in V$ with $i \neq j$ there is a walk from $i$ to $j$, i.e.~$u_1 = i$ and $v_k = j$. A \textbf{tree} is a connected graph without cycles.

The \textbf{adjacency matrix} of a graph $\Gamma = (V, E)$ is the matrix $A = (a_{ij})_{i,j \in V}$ with entries
\[
	a_{i j} \coloneqq \begin{cases}
	1, & \textrm{if } (i,j) \in E \\
	0, & \textrm{otherwise.}
	\end{cases}
\]

\subsection{Symmetries of Graphs}

An \textbf{automorphism} of a graph $\Gamma = (V, E)$ is a bijection $\sigma : V \rightarrow V$ that preserves adjacency and non-adjacency, i.e.~$(i,j) \in E$ if and only if $(\sigma(i), \sigma(j)) \in E$ for all $i, j \in V$. The set of all automorphisms of $\Gamma$ forms a group $\Aut(\Gamma)$ under composition and is called the \textbf{automorphism group} of $\Gamma$. It can be identified with a subgroup of the symmetric group $S_n$ (where $n = \#V$) which can in turn be embedded as permutation matrices in $M_n(\C)$. The automorphism group then has a nice description in terms of the adjacency matrix $A$ of $\Gamma$:
\begin{equation*}
    \Aut(\Gamma) = \{ \sigma \in S_n \mid \sigma A = A \sigma \} \subset S_n
\end{equation*}

We call a graph $\Gamma$ \textbf{symmetric}, if there exists a non-trivial automorphism of $\Gamma$, and \textbf{asymmetric} otherwise.

\subsection{Compact Matrix Quantum Groups}

Compact matrix quantum groups were first defined by Woronowicz \cite{woronowicz1987compact} in an attempt to generalize compact matrix groups. As a general reference for this we refer the reader to \cite{timmermann2008invitation}, \cite{neshveyev2013compact} or \cite{weber2017introduction}.

\begin{defn}
	A \textbf{compact matrix quantum group (CMQG)} is a pair $(B, u)$ where $B$ is a unital $C^*$-algebra and $u = (u_{ij})_{i,j=1}^n$ is a matrix with entries in $B$ such that:
	\begin{enumerate}[(i)]
		\item $B$ is generated (as a $C^*$-algebra) by the entries of $u$.
		\item The $*$-homomorphism $\Delta : B \rightarrow B \otimes B$, $u_{ij} \mapsto \sum_{k=1}^n u_{ik} \otimes u_{kj}$ exists.
		\item The matrix $u$ and its transpose $u^t$ are invertible.
	\end{enumerate}
\end{defn}

\begin{ex}
	Let $G \subseteq \GL_n(\C)$ be a compact matrix group and let $(u_{ij})_{i,j=1}^n$ be the coordinate functions on $G$, i.e.
	\begin{align*}
	u_{ij} : G &\rightarrow \mathbb{C} \\
	B = (b_{kl})_{k,l=1}^n &\mapsto b_{ij}
	\end{align*}
	Then the pair $(C(G), u = (u_{ij})_{i,j=1}^n)$ where $C(G)$ is the algebra of continuous complex-valued functions on $G$ is a compact matrix quantum group. Moreover, all compact matrix quantum groups $(B, u)$ with commutative $C^*$-algebra $B$ are of this form (see \cite[Proposition~6.1.11]{timmermann2008invitation}). So under the identification of $G$ with $(C(G), u)$, compact matrix quantum groups generalize classical compact matrix groups.
\end{ex}

\begin{ex}
	Another very important example of a compact matrix quantum group is the quantum symmetric group $S_n^+$ defined by Wang in \cite{wang1998quantum}. It is given by the universal unital $C^*$-algebra
	\[
	C(S_n^+) \coloneqq C^*(u_{ij}, \ 1 \leq i,j \leq n \mid u_{ij}^* = u_{ij}^2 = u_{ij}, \ \sum_{k=1}^n u_{ik} = \sum_{k=1}^n u_{kj} = 1 \ \forall \ i, j)
	\]
	and it can be interpreted as a quantum analogue of the classical symmetric group $S_n$. Indeed, we have a surjective $*$-homomorphism
	$\phi : C(S_n^+) \rightarrow C(S_n)$ sending generators to coordinate functions, thus we can think of $S_n$ as a quantum subgroup of $S_n^+$. The map $\phi$ is an isomorphism for $n=1,2,3$ (so we have $S_n^+ = S_n$ for $n=1,2,3$) but for $n\geq4$, $C(S_n^+)$ is non-commutative and infinite-dimensional -- we then have more quantum permutations than ordinary permutations.
\end{ex}

\subsection{Quantum symmetries of graphs}

In 2003, Bichon \cite{bichon2003quantum} gave a definition of a quantum automorphism group of a finite graph. It was modified by Banica \cite{banica2005quantum} in 2005. See also \cite{schmidt2018quantum} for more on quantum symmetries of graphs.

\begin{defn}\cite{banica2005quantum}
Let $\Gamma = (V, E)$ be a graph on $n$ vertices without multiple edges and without loops and let $A$ be its adjacency matrix. The \textbf{quantum automorphism group} $\QAut(\Gamma)$ of $\Gamma$ is defined as the compact matrix quantum group with $C^*$-algebra
\begin{equation*}
    C(\QAut(\Gamma)) \coloneqq \faktor{C(S_n^+)}{\langle u A = A u\rangle}
\end{equation*}
We have the inclusions $\Aut(\Gamma) \subset \QAut(\Gamma) \subset S_n^+$ and we say that $\Gamma$ has \textbf{quantum symmetry}, if the first inclusion is strict, i.e.~if and only if $C(\QAut(\Gamma))$ is non-commutative.
\end{defn}

\section{The existence of two cherries}
In this section we show that almost all trees contain at least two cherries. This is probably well-known to experts, but in order to keep this article self-contained, we present a direct proof. We use this in the next section to prove our main result about the quantum symmetries of trees.

\begin{defn}
Let $\Gamma = (V,E)$ be a graph. A triple $(u_1, u_2, v)$ of vertices $u_1, u_2, v \in V$ is called a \textbf{cherry}, if
    \begin{enumerate}[(i)]
        \item $u_1$, $u_2$ and $v$ are pairwise distinct,
        \item $u_1$ and $u_2$ are adjacent to $v$,
        \item $u_1$ and $u_2$ have degree $1$ and
        \item $v$ has degree 3.
    \end{enumerate}
\end{defn}

\begin{rem}
If a graph contains a cherry $(u_1, u_2, v)$, then it admits the non-trivial automorphism that swaps $u_1$ and $u_2$ and fixes any other vertex. Hence the graph has symmetry.
\end{rem}

Erd\H{o}s and R\'{e}nyi showed in \cite{Erdos1963} that almost all trees contain at least one cherry, implying that almost all trees are symmetric. So one can rephrase their result as follows.

\begin{thm}\cite{Erdos1963}\label{thm:er1}
Almost all trees contain at least one cherry, in the sense that
\[
    \lim_{n \rightarrow \infty} \PP [C_n \geq 1] = 1
\]
where $C_n$ is the number of cherries in a tree that is drawn uniformly at random from the set of all trees on $n$ vertices.
\end{thm}

\begin{cor}
Almost all trees have symmetry.
\end{cor}


We will now show that even
\[
    \lim_{n \rightarrow \infty} \PP [C_n \geq 2] = 1
\]
holds in the above setting. The proof of this is only a slight modification of the proof of Theorem~\ref{thm:er1}. Note however that we use a somewhat different notion of cherries than Erd\H{o}s and R\'{e}nyi. In their definition, the requirement (iv) from ours is missing. This changes the formulas in the subsequent proofs by a small degree in comparison with the original arguments in \cite{Erdos1963}.

We begin by fixing some notation. For $n \in \N$ let $T_n$ be a tree on $n$ vertices and denote these vertices by $v_1, \ldots, v_n$. For every choice of indices $i_1, i_2, j \in \{1, \ldots, n\}$ we define
\begin{equation*}
    \epsilon_{i_1, i_2, j}(T_n) \coloneqq \begin{cases*}
        1 \quad & if $(v_{i_1}, v_{i_2}, v_j)$ is a cherry in $T_n$ \\
        0 & otherwise.
    \end{cases*}
\end{equation*}
We equip the set of all labelled trees on $n$ vertices with the uniform probability measure, turning $\epsilon_{i_1, i_2, j}$ into a random variable.

Note furthermore that by Cayley's formula, the number of labelled trees on $n$ vertices is $n^{n-2}$. We will use this fact several times in the following proofs.

\begin{lem}\label{lem:expectation1}
We have
\[
    \E[\epsilon_{i_1, i_2, j}] = \frac{(n-3)^{n-4}}{n^{n-2}}
\]
for all pairwise distinct $i_1, i_2, j \in \{1, \ldots, n\}$.
\end{lem}
\begin{proof}
Since
\[
    \E[\epsilon_{i_1, i_2, j}] = \frac{|\{\text{trees on } n \text{ vertices with a cherry at } (i_1, i_2, j)\}|}{|\{\text{trees on } n \text{ vertices}\}|}
\]
we only have to calculate the numerator, as the denominator is $n^{n-2}$ by Cayley's formula.

Let $T = (V,E)$ be a tree on $n-3$ vertices labelled with $\{1, \ldots, n\} \setminus \{i_1, i_2, j\}$. By attaching a cherry $(v_{i_1}, v_{i_2}, v_{j})$ at any vertex $u \in V$ we can construct a tree on $n$ vertices with a cherry at $(i_1, i_2, j)$. On the other hand, any tree on $n$ vertices with a cherry at $(i_1, i_2, j)$ can be constructed in this way. Since $T$ has $n-3$ vertices, we have $n-3$ possibilities for choosing $u$, thus there are $(n-3)(n-3)^{n-5} = (n-3)^{n-4}$ trees on $n$ vertices with a cherry at $(i_1, i_2, j)$, so the claim follows.
\end{proof}

\begin{lem}\label{lem:expectation2}
Let $n \geq 5$ and $i_1, i_2, j_1, i_3, i_4, j_2 \in \{1, \ldots, n\}$. We have:
\[
    \E[\epsilon_{i_1, i_2, j_1}\epsilon_{i_3, i_4, j_2}] =
    \begin{cases*}
        \frac{(n-6)^{n-6}}{n^{n-2}} \quad & if $i_1,i_2,i_3,i_4,j_1,j_2$ are all different \\
        \frac{(n-3)^{n-4}}{n^{n-2}} & if $i_1, i_2, j_1$ are all different\\
        & and $j_1 = j_2$ and $\{i_1, i_2\} = \{i_3, i_4\}$ \\
        0 & otherwise
    \end{cases*}
\]
\end{lem}
\begin{proof}
Similarly as above, we only have to calculate the numerator of
\[
    \E[\epsilon_{i_1, i_2, j_1}\epsilon_{i_3, i_4, j_2}] = \frac{|\{\text{trees on } n \text{ vertices with cherries at } (i_1, i_2, j_1) \text{ and } (i_3, i_4, j_2)\}|}{|\{\text{trees on } n \text{ vertices}\}|}
\]
So let $T = (V,E)$ be a tree with vertices labelled with $\{1,\ldots,n\} \setminus \{j_1,i_1,i_2,j_2,i_3,i_4\}$.

In the case that all labels $j_1,i_1,i_2,j_2,i_3,i_4$ are different from each other, we can attach cherries $(v_{i_1},v_{i_2},v_{j_1})$ and $(v_{i_3},v_{i_4},v_{j_2})$ at any two vertices $u_1$ and $u_2$ of $\Gamma$ and thereby construct a tree on $n$ vertices with two cherries at $(i_1,i_2,j_1)$ and $(i_3,i_4,j_2)$. On the other hand, every tree on $n$ vertices with two cherries at $(i_1,i_2,j_1)$ and $(i_3,i_4,j_2)$ can be constructed in this way. Since $T$ has $n-6$ vertices, we have $n-6$ possibilities for choosing $u_1$ and $u_2$ respectively. Thus there are $(n-6)(n-6)(n-6)^{n-8} = (n-6)^{n-6}$ trees on $n$ vertices with two cherries at $(i_1,i_2,j_1)$ and $(i_3,i_4,j_2)$.

In the case that $j_1, i_1, i_2$ are distinct, $j_1 = j_2$ and $\{i_1, i_2\} = \{i_3, i_4\}$, we can conclude as in the case of three labels (see Lemma~\ref{lem:expectation1}) that the number of trees on $n$ vertices with a cherry at $(i_1,i_2,j_1)$ is $(n-3)^{n-4}$.

In all other cases, there is no tree on $n$ vertices with cherries at $(i_1,i_2,j_1)$ and $(i_3,i_4,j_2)$.
\end{proof}

As in Theorem~\ref{thm:er1} we denote by $C_n(T_n)$ the number of cherries in the tree $T_n$ on $n$ vertices. With the notation from above we can express this as
\begin{equation}
    C_n(T_n) = \sum_{j=1}^n \sum_{\substack{i_1=1\\i_1\neq j}}^n \sum_{\substack{i_2=1\\i_2\neq j}}^{i_1 - 1} \epsilon_{i_1, i_2, j}(T_n) \label{num_cherry_sum}
\end{equation}

\begin{lem}\label{lem:expectation_Cn}
The expectation of $C_n$ is
\[
    \E[C_n] = \frac{1}{2} n(n-1)(n-2)(n-3)^{n-4} \frac{1}{n^{n-2}} = \frac{n}{2e^3} + O(1)
\]
\end{lem}
\begin{proof}
The number of $3$-tuples $(j, i_1, i_2) \in \{1, \ldots, n\}^3$ such that all entries are distinct is $n(n-1)(n-2)$. The further condition that $i_2 < i_1$ halves this number, so the expression in Equation \eqref{num_cherry_sum} has $\frac{n(n-1)(n-2)}{2}$ summands. Hence, by Lemma~\ref{lem:expectation1}:
\begin{align*}
    \E[C_n] &= \sum_{j=1}^n \sum_{\substack{i_1=1\\i_1\neq j}}^n \sum_{\substack{i_2=1\\i_2\neq j}}^{i_1 - 1} \E[\epsilon_{i_1, i_2, j}] = \sum_{j=1}^n \sum_{\substack{i_1=1\\i_1\neq j}}^n \sum_{\substack{i_2=1\\i_2\neq j}}^{i_1 - 1} \frac{(n-3)^{n-4}}{n^{n-2}}\\
    &= \frac{n(n-1)(n-2)}{2} \frac{(n-3)^{n-4}}{n^{n-2}}
\end{align*}
It remains to show that this is asymptotically $\frac{n}{2e^3} + O(1)$, i.e. we have to show that the following expression is bounded:
\begin{align}
    &\frac{n(n-1)(n-2)}{2} \frac{(n-3)^{n-4}}{n^{n-2}} - \frac{n}{2e^3} = \frac{n}{2}\left( (n-1)(n-2)\frac{(n-3)^{n-4}}{n^{n-2}} - e^{-3}\right) \nonumber\\
    &= \frac{n}{2} \left( \left(\frac{n-3}{n}\right)^{n-2} - e^{-3} + 3 \frac{(n-3)^{n-3}}{n^{n-2}} + 2 \frac{(n-3)^{n-4}}{n^{n-2}} \right) \nonumber\\
    &= \frac{1}{2} \left( n \left( \left(1 - \frac{3}{n}\right)^{n-2} - e^{-3} \right) + 3 \frac{(n-3)^{n-3}}{n^{n-3}} + 2 \frac{(n-3)^{n-4}}{n^{n-3}} \right) \label{eq:asymp_of_ECn}
\end{align}
The second and third summand in Term~\eqref{eq:asymp_of_ECn} converge to $3e^{-3}$ and $0$ respectively as $n \rightarrow \infty$. So we have to check whether the first summand is bounded. For large $n$ we have\footnote{This calculation can be made rigorous by explicitly writing out a full Taylor expansion of the logarithm and keeping track of the limit definition of the exponential function. As this is tedious and does not add any further insight, we leave out the details.}:
\begin{align*}
    &\left( 1 - \frac{3}{n} \right)^{n-2} - e^{-3}
     = \exp \left((n-2) \log \left(1-\frac{3}{n} \right) \right) - e^{-3}\\
    &\approx \exp \left( n \left( \frac{-3}{n} - \frac{9}{2n^2} \right) \right) - e^{-3}
    = e^{-3} e^{-\frac{9}{2n}} - e^{-3} \approx e^{-3} \left( 1- \frac{9}{2n} \right) - e^{-3} = -\frac{9e^{-3}}{2n}
\end{align*}
Hence $n \left( \left(1 - \frac{3}{n}\right)^{n-2} - e^{-3} \right)$ is bounded and thus $\frac{n(n-1)(n-2)}{2} \frac{(n-3)^{n-4}}{n^{n-2}} = \frac{n}{2e^3} + O(1)$ as claimed.
\end{proof}

We now want to calculate the variance of $C_n$. For this we need the second moment.

\begin{lem}\label{lem:second_moment_Cn}
The second moment of $C_n$ is
\[
    \E[C_n^2] = \left( \frac{1}{4} \frac{n!}{(n-6)!} (n-6)^{n-6} + n(n-1)(n-2)(n-3)^{n-4} \right) \frac{1}{n^{n-2}} = \frac{n^2}{4e^6} + O(n)
\]
\end{lem}
\begin{proof}
Let $T_n$ be a random tree on $n$ vertices. We first compute using Equation~\ref{num_cherry_sum}:
\begin{align*}
    \E[C_n^2] &= \sum_{j_1 = 1}^n \sum_{\substack{i_1=1\\i_1\neq j_1}}^n \sum_{\substack{i_2=1\\i_2 \neq j_1}}^{i_1-1} \sum_{j_2=1}^n \sum_{\substack{i_3=1\\i_3 \neq j_2}}^n \sum_{\substack{i_4=1\\i_4 \neq j_2}}^{i_3-1} \E[\epsilon_{i_1, i_2, j_1} \epsilon_{i_3, i_4, j_2}]
\end{align*}
To apply the formulas from Lemma~\ref{lem:expectation2} we split this sum into the two cases where $j_1, i_1, i_2, j_2, i_3, i_4$ are all different and where either $j_1 = j_2$ and $i_1 = i_3$, $i_2 = i_4$ or $j_1 = j_2$ and $i_1 = i_4$, $i_2 = i_3$ and $j_1, i_1, i_2$ are different.
\begin{align}
    \E[C_n^2] &= \sum_{j_1=1}^n \sum_{\substack{i_1=1\\i_1\neq j_1}}^n \sum_{\substack{i_2=1\\i_2\neq j_1}}^{i_1-1} \sum_{\substack{j_2=1\\j_2\neq j_1\\j_2\neq i_1\\j_2\neq i_2}}^n \sum_{\substack{i_3=1\\i_3\neq j_1\\i_3\neq i_1\\i_3\neq i_2\\i_3\neq j_2}}^n \sum_{\substack{i_4=1\\i_4\neq j_1\\i_4\neq i_1\\i_4\neq i_2\\i_4\neq j_2}}^{i_3-1} \E[\epsilon_{i_1, i_2, j_1} \epsilon_{i_3, i_4, j_2}] \label{eq:big_sum} \\
    &+ 2 \sum_{j_1=1}^n \sum_{\substack{i_1=1\\i_1\neq j_1}}^n \sum_{\substack{i_2=1\\i_2\neq j_1}}^{i_1-1} \E[\epsilon_{i_1, i_2, j_1}^2] \label{eq:small_sum}
\end{align}
The number of $6$-tuples $(j_1,i_1,i_2,j_2,i_3,i_4) \in \{1,\ldots,n\}^6$ such that all entries are different is $\frac{n!}{(n-6)!}$. Each of the further conditions that $i_2 < i_1$ and $i_4 < i_3$ halves this number. So the expression in Term~\eqref{eq:big_sum} has $\frac{1}{4}\frac{n!}{(n-6)!}$ summands. By a similar argument, the expression in Term~\eqref{eq:small_sum} has $\frac{n(n-1)(n-2)}{2}$ summands, therefore by Lemma~\ref{lem:expectation2} and analogous calculations as in the proof of Lemma~\ref{lem:expectation_Cn}:
\begin{align*}
    \E[C_n^2] &= \frac{1}{4}\frac{n!}{(n-6)!}\frac{(n-6)^{n-6}}{n^{n-2}} + n(n-1)(n-2) \frac{(n-3)^{n-4}}{n^{n-2}}\\
    &= \frac{n^2}{4e^6} + O(n)
\end{align*}
\end{proof}

\begin{lem}\label{lem:variance_Cn}
The following holds for $C_n$:

\begin{minipage}{.45\linewidth}
    \begin{flushleft}
        \begin{align*}
            (i)& \ \frac{\E[C_n^2]}{\E[C_n]^2} \xrightarrow{n \rightarrow \infty} 1\\
            (iii)& \ \frac{\Var[C_n]}{\E[C_n]^2} \xrightarrow{n \rightarrow \infty} 0
        \end{align*}
    \end{flushleft}
\end{minipage}
\hfill
\begin{minipage}{.45\linewidth}
    \begin{flushright}
         \begin{align*}
            (ii)& \ \frac{\E[(C_n-1)^2]}{\E[C_n-1]^2} \xrightarrow{n \rightarrow \infty} 1\\
            (iv)& \ \frac{\Var[C_n-1]}{\E[C_n-1]^2} \xrightarrow{n \rightarrow \infty} 0
        \end{align*}
    \end{flushright}
\end{minipage}

\end{lem}
\begin{proof}
\begin{enumerate}[(i)]
    \item Using Lemma~\ref{lem:expectation_Cn} and Lemma~\ref{lem:second_moment_Cn} we calculate:
    \begin{align*}
        \frac{\E[C_n^2]}{\E[C_n]^2} &= \frac{\frac{n^2}{4e^6} + O(n)}{\left( \frac{n}{2e^3}n + O(1) \right)^2} = \frac{\frac{1}{4e^6} n^2 + O(n)}{\frac{1}{4e^6} n^2 + O(n)} \xrightarrow{n \rightarrow \infty} 1
    \end{align*}
    \item Again, by using Lemma~\ref{lem:expectation_Cn} and Lemma~\ref{lem:second_moment_Cn} we get:
    \begin{align*}
        \frac{\E[(C_n-1)^2]}{\E[C_n-1]^2} &= \frac{\E[C_n^2] - 2\E[C_n] + 1}{\E[C_n]^2 - 2\E[C_n] + 1} = \frac{ \frac{n^2}{4e^6} + O(n) - \frac{n}{e^3} + O(1) + 1 }{ \left( \frac{n}{2e^3} + O(1) \right)^2 - \frac{n}{e^3} + O(1) + 1 } \\
        &= \frac{\frac{1}{4e^6} n^2 + O(n)}{\frac{1}{4e^6} n^2 + O(n)} \xrightarrow{n \rightarrow \infty} 1
    \end{align*}
    \item  + (iv) As $\frac{\Var[X]}{\E[X]^2} = \frac{\E[X^2]}{\E[X]^2} - 1$, we obtain (iii) and (iv) from (i) and (ii) respectively.
\end{enumerate}
\end{proof}

\begin{thm}\label{two_cherries}
Almost all trees have at least two cherries, i.e.
\[
    \lim_{n \rightarrow \infty} \PP[C_n \geq 2] = 1
\]
\end{thm}
\begin{proof}
Using Chebyshev's inequality and Lemma~\ref{lem:variance_Cn} we have that
\begin{equation*}
    \PP[C_n = 0] \leq \PP\big[\lvert C_n - \E[C_n] \rvert \geq \E[C_n]\big] \leq \frac{\Var[C_n]}{\E[C_n]^2} \xrightarrow{n \rightarrow \infty} 0
\end{equation*}
as well as
\begin{align*}
    \PP[C_n = 1] = \PP[C_n - 1 = 0] \leq \frac{\Var[C_n - 1]}{\E[C_n - 1]^2} \xrightarrow{n \rightarrow \infty} 0
\end{align*}
So $\PP[C_n \geq 2] = 1 - \PP[C_n = 0] - \PP[C_n = 1] \xrightarrow{n \rightarrow \infty} 1$ which completes the proof.
\end{proof}

\section{Asymptotics of (quantum) symmetries of graphs and trees}

From Theorem~\ref{two_cherries} and the following criterion from \cite{schmidt2018foldedcube}, we may derive that almost all trees have quantum symmetry.

Let $\Gamma = (V,E)$ be a graph and let $\sigma : V \rightarrow V$ be an automorphism of $\Gamma$. The set
\begin{equation*}
    \{ v \in V \mid \sigma(v) \neq v \}
\end{equation*}
is called the \textbf{support} of $\sigma$.

\begin{prop}\cite[Theorem~2.2]{schmidt2018foldedcube}
\label{schmidt_criterion}
Let $\Gamma$ be a graph. If there exist two non-trivial automorphisms $\sigma, \tau$ of $\Gamma$ with disjoint support, then $\Gamma$ has quantum symmetry.
\end{prop}

This also makes the importance of cherries clear. If a graph contains a cherry, it has a non-trivial automorphism of order two which swaps the two vertices of degree one in the cherry. So the graph has symmetry. If a graph has two cherries, it has two non-trivial disjoint automorphisms of order two. So by the above proposition, it has quantum symmetry. From this, we can conclude our main result, building on Theorem~\ref{two_cherries}.

\begin{thm}\label{thm:trees_qsym}
    Almost all trees have quantum symmetry.
\end{thm}

For the convenience of the reader, let us briefly describe the arguments of \cite{lupini2017nonlocal} for the non-existence of quantum symmetry for graphs.

A subset $\mathcal{A} \subseteq M_n(\C)$ is called a \textbf{coherent algebra} if it is a selfadjoint unital subalgebra of $M_n(\C)$ both with respect to ordinary matrix multiplication as well as with respect to entrywise matrix multiplication (Schur product).

A class of examples of coherent algebras can be obtained from group actions:
Let $G$ be a group acting on a finite set $X$ with $n$ elements and let $R_1, \ldots, R_s$ be the orbits of the induced action of $G$ on $X \times X$ (these are sometimes called \textit{orbitals}). For each $R_i$ ($i \in \{ 1, \ldots, s\}$) we define its characteristic matrix $A^{(i)}$ as
\[
    A^{(i)}_{x y} \coloneqq \begin{cases}
                                1, &\textrm{if } (x, y) \in R_i\\
                                0, &\textrm{otherwise.}
                            \end{cases}
\]
Then the linear span of these matrices is a coherent algebra.

One can also associate a coherent algebra to a graph $\Gamma$ by considering the coherent algebra generated by its adjacency matrix, i.e.~the intersection of all coherent algebras containing the adjacency matrix of $\Gamma$. This is then called the \textbf{coherent algebra of $\Gamma$} and denoted by $\mathcal{CA}(\Gamma)$.

It is a well-known fact that the coherent algebra of a graph provides a sufficient criterion for a graph to be asymmetric:

\begin{prop}\label{ca_aut}
Let $\Gamma = (V,E)$ be a graph on $n$ vertices. Then we have:
\[
    \mathcal{CA}(\Gamma) = M_n(\C) \Rightarrow \Aut(\Gamma) = \{ \id \}
\]
\end{prop}

This is useful from a computational point of view since the coherent algebra of a graph can be computed in polynomial time using the two-dimensional Weisfeiler-Leman algorithm (for details on this see \cite{fuerer_WL}), whereas it is generally hard to compute the automorphism group of a graph or even to check whether or not it is trivial.

One of the key insights of Lupini, Man\v{c}inska and Roberson was that we can strengthen Proposition~\ref{ca_aut} to the following.

\begin{prop}\cite{lupini2017nonlocal}\label{ca_qaut}
Let $\Gamma = (V,E)$ be a graph on $n$ vertices. Then we have:
\[
    \mathcal{CA}(\Gamma) = M_n(\C) \Rightarrow \QAut(\Gamma) = \{ \id \}
\]
\end{prop}

Now, Babai and Kucera proved  that almost all graphs have maximal coherent algebra.

\begin{prop}\cite{babai1979}\label{almost_all_ca}
Almost all graphs $\Gamma$ satisfy $\mathcal{CA}(\Gamma) = M_n(\C)$. 
\end{prop}

Summarizing all of the above, we obtain:

\begin{thm} \ \newline \vspace*{-\baselineskip} \enlargethispage{\baselineskip}
	\begin{enumerate}[(i)]
		\item \cite{Erdos1963} Almost all graphs have no symmetry.
		\item \cite{lupini2017nonlocal} Almost all graphs have no quantum symmetry.
		\item \cite{Erdos1963} Almost all trees have symmetry.
		\item Almost all trees have quantum symmetry.
	\end{enumerate}
\end{thm}
\begin{proof}
\begin{enumerate}[(i)]
    \item This follows from Proposition~\ref{almost_all_ca} and Proposition~\ref{ca_aut}.
    \item By combining Proposition~\ref{almost_all_ca} and Proposition~\ref{ca_qaut} we obtain the claimed result.
    \item By Theorem~\ref{two_cherries} almost all trees have in particular \textit{one} cherry. So their automorphism group contains a copy of $\Z_2$.
    \item This is exactly the statement of Theorem~\ref{thm:trees_qsym}.
\end{enumerate}
\end{proof}

\begin{rem}
Note that Lupini, Man\v{c}inska and Roberson in fact showed that $\QAut(\Gamma) = \{\id\}$ for almost all graphs, but this implies in particular that $\QAut(\Gamma) = \Aut(\Gamma)$ for almost all graphs, which is the actual definition of having no quantum symmetry.
\end{rem}

\bibliography{main}{}
\bibliographystyle{alpha}
\end{document}